\documentclass[a4paper]{amsart}

\usepackage{amsmath,amssymb,amsthm,a4wide}

\usepackage{tikz}
\usetikzlibrary{calc,intersections,through,backgrounds}

\usepackage{hyperref,caption}
\usepackage{graphicx}
\usepackage{graphics}

\newtheorem{theorem}{Theorem}
\newtheorem*{thm}{Theorem}
\newtheorem*{lemma}{Lemma}
\newtheorem{corollary}{Corollary}
\newtheorem*{proposition}{Proposition}

\newcommand{\diam}{\operatorname{diam}}

\begin{document}

\title[]{On the location of Maxima of Solutions\\ of Schr\"odinger's equation}
\keywords{Schr\"odinger equation, location of maxima, torsion function, ground state.}
\subjclass[2010]{35B38, 35J05 (primary)}

\author[]{Manas Rachh}
\address[Manas Rachh]{Applied Mathematics Program, Yale University, New Haven, CT 06510, USA}
\email{manas.rachh@yale.edu}

\author[]{Stefan Steinerberger}
\address[Stefan Steinerberger]{Department of Mathematics, Yale University, New Haven, CT 06510, USA}
\email{stefan.steinerberger@yale.edu}

\begin{abstract} 
We prove an inequality with applications to solutions of 
the Schr\"odinger equation. There is a universal constant $c>0$, 
such that if $\Omega \subset \mathbb{R}^2$ is simply connected, 
$u:\Omega \rightarrow \mathbb{R}$ vanishes on the boundary $\partial \Omega$, 
and $|u|$ assumes a maximum in $x_0 \in \Omega$, then 
$$   \inf_{y \in \partial \Omega}{ \| x_0 - y\|}  \geq c \left\| \frac{\Delta u}{u} \right\|^{-1/2}_{L^{\infty}(\Omega)}.$$
It was conjectured by P\'olya \& Szeg\H{o} (and proven, independently, by Makai and Hayman) that a membrane vibrating at frequency $\lambda$
contains a disk of size $\sim \lambda^{-1/2}$. Our inequality implies a refined result: the point on the membrane that achieves the maximal amplitude is at distance $\sim \lambda^{-1/2}$ from the boundary. We also give an extension to higher dimensions (generalizing results of 
Lieb, and Georgiev \& Mukherjee): if $u$ solves $-\Delta u = Vu$ on $\Omega \subset \mathbb{R}^n$ with Dirichlet boundary conditions, then
the ball $B$ with radius $\sim \|V\|_{L^{\infty}(\Omega)}^{-1/2}$ centered at the point in which $|u|$ assumes a maximum is almost fully contained in $\Omega$ in the sense that $|B \cap \Omega| \geq 0.99 |B|.$
\end{abstract}

\maketitle

\vspace{-20pt}

\section{Introduction}

\subsection{Introduction} The purpose of this paper is to state and prove an inequality saying that on a simply connected domain $\Omega \subset \mathbb{R}^2$, a nonzero solution of the partial differential equation
\begin{align*}
- \Delta u &= V u \hspace{2pt}  \quad \mbox{in}~\Omega \\
 u &= 0     \qquad \hspace{0pt} \mbox{on}~\partial\Omega, 
\end{align*}
cannot have its maxima or minima close to the boundary unless $\|V\|_{L^{\infty}}$ is large.
Dirichlet conditions are necessary because under Neumann conditions one expects the maximum to be assumed on the boundary (this is
the hot spots conjecture, see \cite{burd2}).

\begin{theorem}
There is a constant $c>0$ such that for all simply-connected $\Omega \subset \mathbb{R}^2$ and all $u:\Omega \rightarrow \mathbb{R}$, the following holds: if $u$ vanishes on the boundary $\partial \Omega$ and $|u(x_0)| = \| u \|_{L^{\infty}(\mathbb{R})}$, then
$$   \inf_{y \in \partial \Omega}{ \| x_0 - y\|}  \geq c \left\| \frac{\Delta u}{u} \right\|^{-1/2}_{L^{\infty}(\Omega)}.$$
\end{theorem}
We interpret the $L^{\infty}-$term as $\infty$ if the quantity is unbounded or not defined. If $|u|$ does not assume a global maximum, the statement is empty. 
A careful analysis
of the proof shows that the result applies to positive maxima and negative minima 
($\forall x \in \Omega: u(x) \leq u(x_0)$ and $u(x_0) > 0$, and 
$\forall x \in \Omega: u(x) \geq u(x_0)$ and $u(x_0) < 0$, respectively). 
The exponent in the above inequality is sharp: if $\Omega = [0,\pi]^2$ and 
$ u(x,y) = \sin{(n x)}\sin{(y)}$ for $n \in \mathbb{N}$, then
$$   \inf_{y \in \partial \Omega}{ \| x_0 - y\|}  = \frac{\pi}{2n},  \qquad \mbox{and} \qquad \left\| \frac{\Delta u}{u} \right\|^{}_{L^{\infty}(\Omega)} = n^2 + 1.$$

\begin{center}
\begin{figure}[h!]
\begin{tikzpicture}[scale=0.7]
\draw [thick] (0,-1) to[out=340, in=120] (2,-0.5) to[out=300, in=210] (3,-1)  to[out=30, in=270] (4,0)  to[out=90, in=0] (3,1)   to[out=180, in=30] (2,0)   to[out=210, in=330] (1,0.5)    to[out=150, in=90] (-1,0) 
 to[out=270, in=160] (0,-1);
\filldraw (0.2,-0.15) circle (0.04cm);
\node at (0.5, -0.2) {$x$};
\filldraw (3,-0.1) circle (0.04cm);
\node at (3.3, -0.2) {$y$};
\node at (4.3, -0.8) {$\Omega$};
\end{tikzpicture}
\caption{Unless $\| (\Delta u)/u \|_{L^{\infty}(\Omega)}$ is large, extrema of $u$ are close to either $x$ or $y$.}
\end{figure}
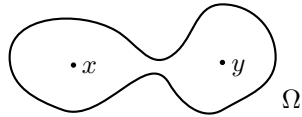
\end{center}

\subsection{Related results}
It is instructive to specialize the result to the case of the first 
Laplacian eigenfunction $-\Delta u = \lambda_1(\Omega) u$ with Dirichlet 
boundary conditions. 
A problem that was first raised in 1951 by P\'{o}lya \& Szeg\H{o} in their famous \textit{Isoperimetric Inequalities in Mathematical Physics} \cite{pol}
is whether there exists a constant $c_2>0$ such that for all simply connected domains $\Omega \subset \mathbb{R}^2$
$$ \mbox{inradius}(\Omega) \geq  c_2\cdot  \lambda_1(\Omega)^{-1/2}.$$
Hersch \cite{hersch} showed that the best constant for convex $\Omega \subset \mathbb{R}^2$ is $c_2= \pi/2$. The inequality was then first established by
Makai \cite{makai} in 1965 and, independently, by Hayman \cite{hayman} in 1977 (Makai's result was, for a long time, not well known, see \cite{car2}). The optimal
constant $c_2$ is sometimes called the Hayman-Makai constant or Osserman constant. Quantitative
improvements were given by Osserman \cite{osserman}, Croke \cite{croke} and Protter \cite{protter} -- a substantial refinement was given by Jerison \& Grieser \cite{jerison}. Currently, the best constant in the inequality is $c_2 \geq 0.78$ due to Banuelos \& Carroll \cite{ban2}. We prove a stronger (albeit non-quantitative) result.

\begin{corollary} There exists a universal constant $c_1 > 0$ such that the first Laplacian eigenfunction on a simply connected domain $\Omega \subset \mathbb{R}^2$ assumes its maximum at a distance of
at least $c_1 \cdot \lambda_1(\Omega)^{-1/2}$ from the boundary.
\end{corollary}
The example above already implies $c_1 \leq \pi/2$ and, trivially, $c_1 \leq c_2$.
Existing intricate constructions of Banuelos \& Carroll \cite{ban2} and Brown \cite{brown} give $c_2 \leq 1.446$ and this seems to be close
to optimal for $c_2$. We have a numerical example showing $c_1 \leq 1.37$ (see \S 1.4 below). It is, in principle,
possible to make our lower bound quantitative -- however, a rough estimate (see below) shows that our proof
in its current form will not be able to produce any bound that is better than $c_1 \geq 0.1$, which seems far from the truth. Any result of the form $c_1 \geq 0.78$ is likely to
be difficult as it would also imply an improved bound for the Hayman-Makai-Osserman constant $c_2$ (see \cite{ban2}).

\subsection{Higher dimensions} It was noted by Hayman \cite{hayman} 
that for $\Omega \subset \mathbb{R}^n$ and $n \geq 3$, it
is not possible to bound the inradius depending on 
$\lambda_1(\Omega)$. One can simply remove lines from the
set: this decreases the inradius but does not affect the eigenvalue. 
However, there is a celebrated result of Lieb stating that
one {\em essentially} finds a ball with the desired radius.

\begin{thm}[Lieb, 1983] Let $n \geq 3$ and $\Omega \subset \mathbb{R}^n$ be open. For every $\varepsilon > 0$, there exists a $c(\varepsilon) > 0$ such that
there exists a ball $B$ of radius $c(\varepsilon) \lambda_1(\Omega)^{-1/2}$ for which
$$ |B \cap \Omega| \geq (1- \varepsilon) |B|.$$
\end{thm}

Little seems to be known about the location of the maximum.  The second author \cite{stein1} proved that if  $-\Delta u = \lambda u$ with Dirichlet conditions and one starts Brownian motion where $|u|$ assumes its maximum, then the
likelihood of hitting the boundary within time $t=\lambda^{-1}$ is less than 63.3\% ($= (e-1)/e$).
Georgiev \& Mukherjee \cite{georg} used this fact to refine Lieb's theorem.  

\begin{thm}[Georgiev \& Mukherjee, \cite{georg}] The ball with all the properties in Lieb's theorem can be placed around the point where $|u|$
assumes its maximum.
\end{thm}

We show that the relevant Brownian motion impact inequality is also true for solutions of Schr\"odinger equations and then use the argument of Georgiev \& Mukherjee 
in the more general setup.

\begin{theorem} Let $n \geq 3$, $\Omega \subset \mathbb{R}^n$ be open and 
suppose $u:\Omega \rightarrow \mathbb{R}$ solves $-\Delta u = V u$ with Dirichlet boundary conditions.
For every $\varepsilon > 0$ there exists a function $c(\varepsilon) > 0$ (depending only on the dimension) such that the ball $B$ centered at the maximum of $|u|$ 
$$ \mbox{with radius} \quad c(\varepsilon)\|V\|_{L^{\infty}(\Omega)}^{-1/2} \quad \mbox{satisfies} \qquad  |B \cap \Omega| \geq (1- \varepsilon) |B|.$$
\end{theorem}
This generalizes the results of Lieb and Georgiev \& Mukherjee. It also
refines a result of De Carli \& Hudson \cite{dec}: they show that if $-\Delta u = Vu$ in $\Omega \subset \mathbb{R}^n$ with Dirichlet boundary conditions, then 
$$|\Omega| \cdot \|V\|_{L^{\infty}(\Omega)}^{n/2} \geq c_n \, , 
\qquad \mbox{for some universal constant}~c_n > 0.$$
They also determine the optimal $c_n$ (which is attained for $\Omega$ a ball and $V$ constant). Our result immediately implies
the same universal bound (though without the sharp constant) since $\Omega$ contains large portions of a ball whose radius already implies the desired bound on
the volume.

\subsection{The torsion function.} Our discovery of the inequality was motivated by studying a related problem.
Let $\Omega \subset \mathbb{R}^2$ be a convex and bounded domain and suppose that
\begin{align*}
- \Delta u &= f(u) \quad \mbox{in}~\Omega \\
 u &= 0     \qquad \hspace{4pt} \mbox{on}~\partial\Omega 
\end{align*}
has a positive solution. Here $f$ is assumed to be Lipschitz continuous and restoring: $f(z) > 0$ whenever $z>0$.
Cima \& Derrick \cite{cima1} and Cima, Derrick \& Kalachev \cite{cima2} asked whether the location of the maximum is independent of the nonlinearity $f$.
This was disproven by Benson, Laugesen, Minion and Siudeja \cite{be}: they considered the nonlinearities $f(u) = 1$
and $f(u) = \lambda_1(\Omega)u$ on both the semidisk and the right isoceles triangle and showed the maxima to be in different
locations by explicit computation: however, they are \textit{extremely} close (their distance is of order $\sim 10^{-5}$ assuming $\diam(\Omega) \sim 1$).
We found this fact absolutely striking and it originally motivated the work that resulted in this paper.
Let us introduce the torsion function as the solution of
\begin{align*}
 \Delta v &= -1 \qquad  \hspace{1pt}\mbox{in}~\Omega \\
 v &= 0     \qquad \mbox{on}~\partial\Omega .
\end{align*}
It has recently received a lot of attention as a predictor for where 
Laplacian eigenfunctions localize \cite{may2,may1,stein2}.
The following consequence of the main result is a further instance of the connection between the Laplacian eigenfunction and the torsion function.

\begin{corollary}  There exists a universal constant $c>0$ such that for every bounded, simply connected $\Omega \subset \mathbb{R}^2$ for which the first Laplacian eigenfunction assumes
a global maximum in $x_0 \in \Omega$
$$  v(x_0) \geq c\sup_{x \in \Omega}{v(x)}.$$
\end{corollary}

Generically, and perhaps for all convex domains, the constant seems to be remarkably close to 1.
We illustrate this with an example that was essentially picked at random: let the boundary of
$\Omega$ be given by (see Fig. 2)
$$r(\theta) =1 + 0.25\cos{(t)} + 0.4 \sin{(3t)}.$$

\begin{figure}[h!]
\begin{minipage}{0.49\textwidth}
\begin{center}
\includegraphics[width= 0.8\textwidth]{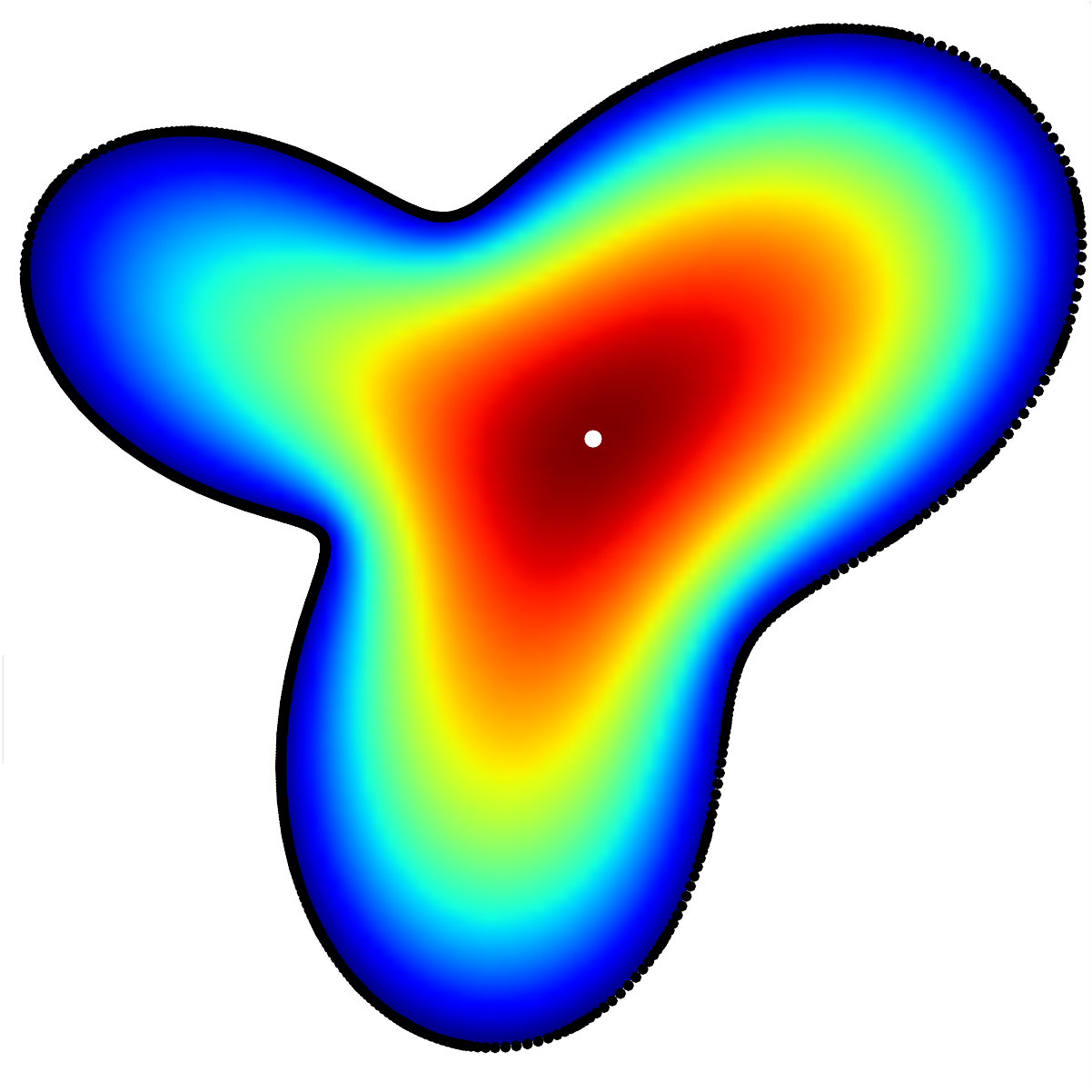}
\end{center}
\end{minipage}
\begin{minipage}{0.49\textwidth}
\begin{center}
\includegraphics[width=0.8\textwidth]{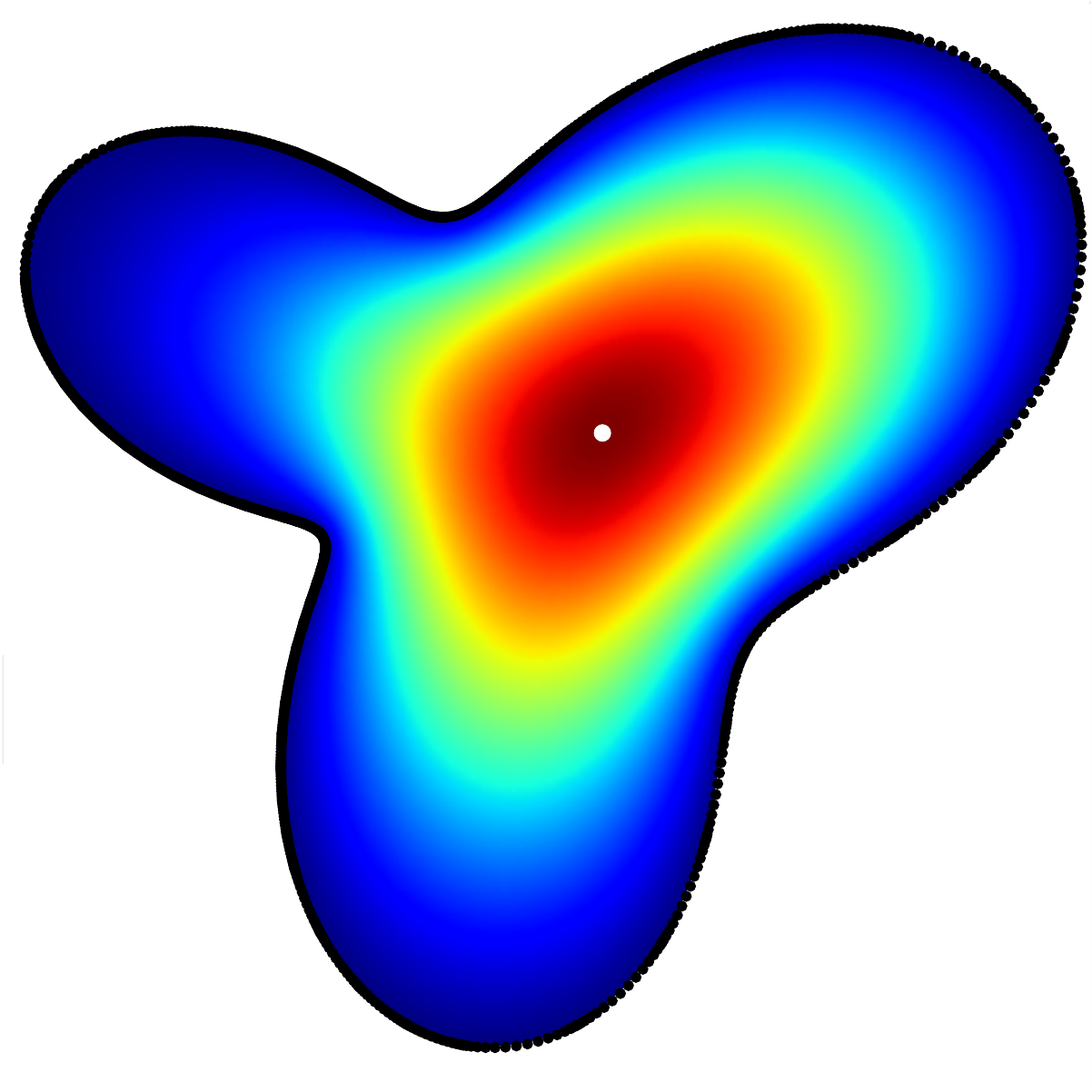}
\end{center}
\end{minipage}
\caption{The torsion function (left) and the first Laplacian eigenfunction (right).}
\end{figure}
The maxima of torsion function and eigenfunction are marked in the picture: they are at a distance of roughly $\sim 0.01 \diam(\Omega)$ from each other
and we have 
$$v(x_0) \sim 0.99948 \sup_{x \in \Omega}{v(x)}.$$
We emphasize that this example was essentially picked at random and one really seems to get comparable results for most domains. 
 The best upper bound we found through extensive numerical simulations is $c \leq 0.976$ (see Fig. 3). The boundary of the corresponding domain is given by
 $$r(\theta) = 1 + 0.49\cos{(2\theta)}.$$
The domain is rather delicate: the Laplacian eigenfunction cannot localize in any of the two balls and is forced to be fairly central while the
torsion function can localize to a stronger extent. Indeed, their maxima are at distance $\sim 0.2 \diam(\Omega)$. This shows that the close proximity of these two maxima
that was observed in \cite{be} is not always the case (and likely a result of the presence of an axis of symmetry and convexity of the domains considered in \cite{be}).   

\begin{figure}[h!]
\begin{minipage}{0.49\textwidth}
\includegraphics[width= 0.99\textwidth]{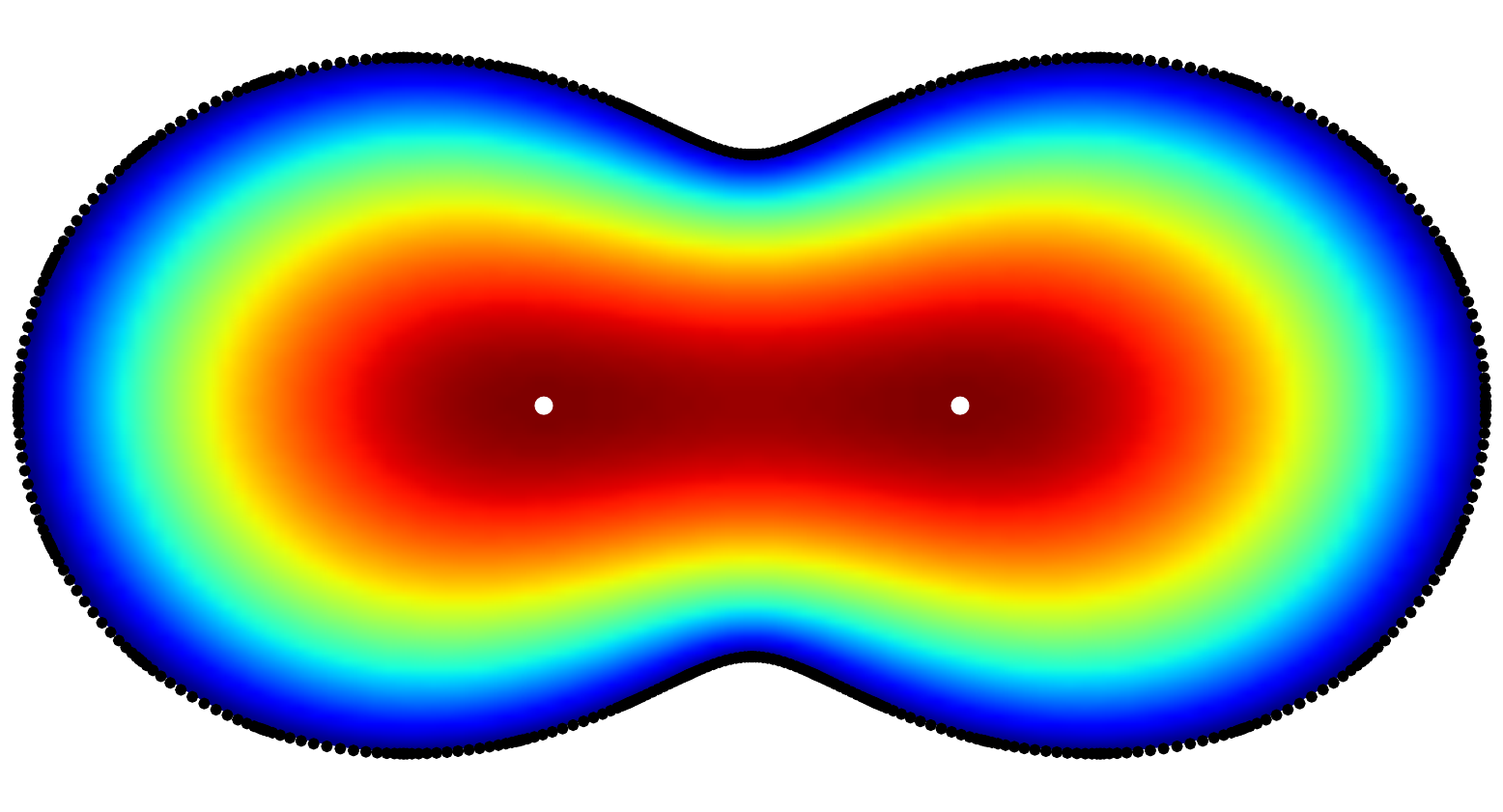}
\end{minipage}
\begin{minipage}{0.49\textwidth}
\includegraphics[width=0.99\textwidth]{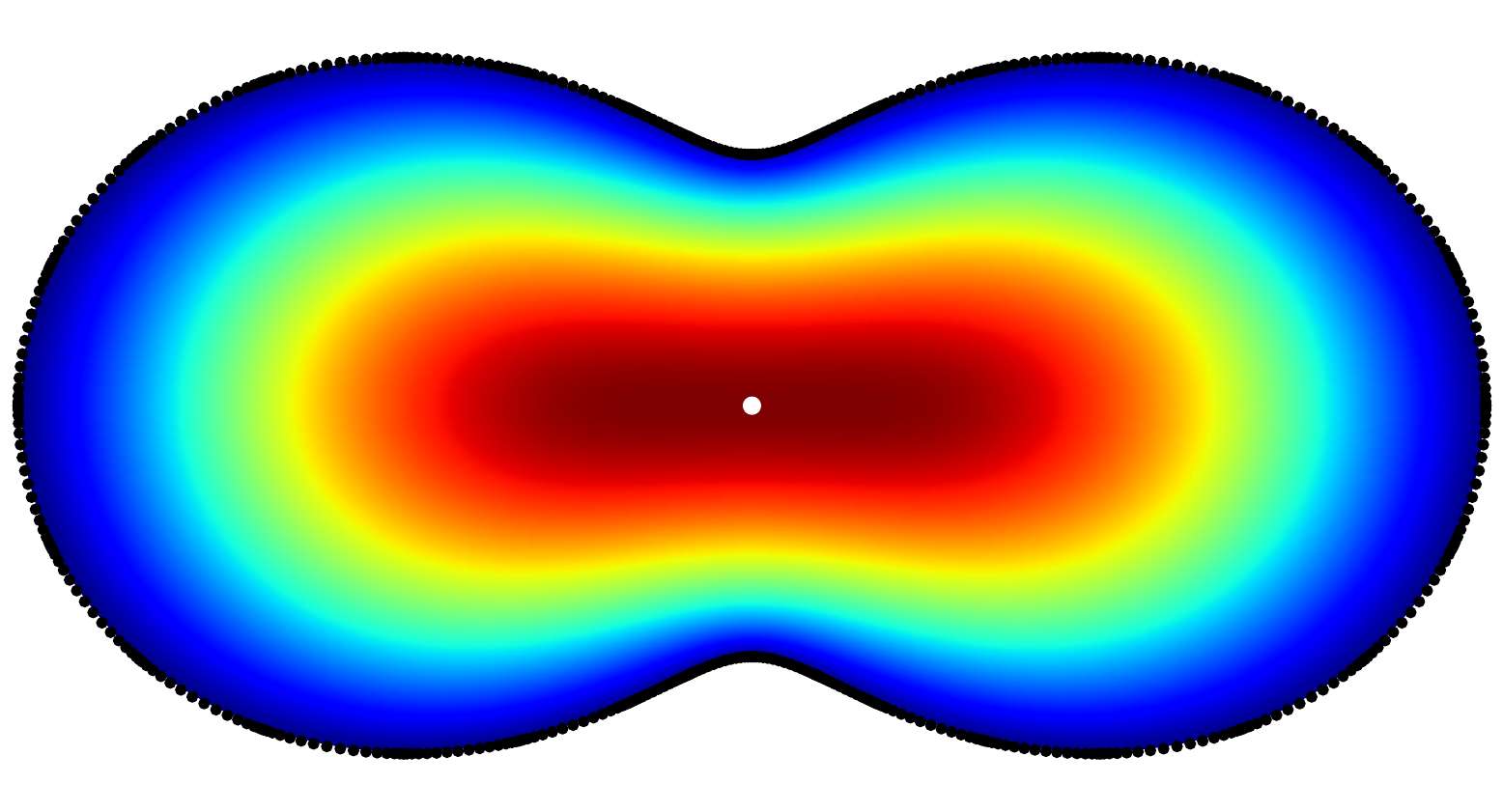}
\end{minipage}
\caption{The torsion function (left) and the first Laplacian eigenfunction (right).}
\end{figure}
We consider these results striking and believe they are worthy of further study.

\begin{quote}
\textbf{Open problem.} What is the optimal constant in
$$  v(x_0) \geq c\sup_{x \in \Omega}{v(x)} \, ,$$
and for which domains is it assumed? What happens on convex domains?
\end{quote}
The same question can be asked for general nonlinear Poisson equations. For the first Laplacian eigenfunction on a 
convex domain in $\mathbb{R}^2$, results with a strong geometric
flavor have been obtained by Brasco, Magnanini \& Salani \cite{brasco} who related its 
location to the `heart' of the convex domain (see \cite{heart}).
The example in Fig. 3 shows that the optimal 
constant $c_1 > 0$ in
$$   \inf_{y \in \partial \Omega}{ \| x_0 - y\|}  \geq c_1 
\left\| \frac{\Delta u}{u} \right\|^{-1/2}_{L^{\infty}(\Omega)} \, , $$
satisfies $c \leq 1.37$ (the distance between the location of the maximum and the boundary is quite a bit smaller
than the inradius ($d(x_0, \partial \Omega) = 0.49 ,\mbox{inrad}(\Omega) \sim 0.71$ and $\lambda_1(\Omega) = 7.785$)). \\

We used high order boundary integral equation methods to compute both the
torsion function and the first eigenfunction with an accuracy of $10^{-6}$ 
for the examples displayed above.
Carefully designed boundary integral equations result 
in well-conditioned linear systems upon discretization and the 
condition number does not grow as the mesh is refined.
Moreover, these methods only require discretization of the boundary for solving 
homogenous elliptic partial differential equations, thereby significantly 
reducing the number of unknowns in the discretized linear system.
For a more detailed discussion of the numerical tools used in this paper, 
we refer to \cite{hao-barnett,helsing}.

\section{Proofs}
The proofs are all relatively short and based on probabilistic interpretation (we refer to Simon \cite{simon} for an introduction). The main idea is
to regard the solution of the elliptic problem $-\Delta u = Vu$ as the time-independent solution of the parabolic problem
\begin{align*}
u_t &= (\Delta + V)u \qquad \mbox{in}~\Omega \\
 u &= 0     \qquad \hspace{37pt} \mbox{on}~\partial\Omega.
\end{align*}
This parabolic problem is then analyzed by using the stochastic interpretation and the Feynman-Kac formula. We will denote a Brownian motion path
in $\mathbb{R}^n$ as $\omega:[0, \infty] \rightarrow \mathbb{R}^n$. The Feynman-Kac formula allows us to represent the solution of the parabolic
problem as
$$ e^{t(\Delta + V)}u(x) = \mathbb{E}_{x}\left(u(\omega(t)) e^{\int_{0}^{t}{V( \omega(z))dz}} \right),$$
where the expectation is taken with respect to all Brownian motions that start at $x$ with the condition that if it they collide with the boundary, then
they stays there for all time (the boundary is `sticky'). Since we know our solution to be time-independent, we get for all $t \geq 0$
$$ u(x) = \mathbb{E}_{x}\left(u(\omega(t)) e^{\int_{0}^{t}{V( \omega(z))dz}} \right).$$

\subsection{Proof of Theorem 1}
\begin{proof}
Assuming
$-\Delta u = V u$ for some $\|V\|_{L^{\infty}} < \infty$, the Feynman-Kac formula implies
$$ u(x) = \mathbb{E}_{x}\left(u(\omega(t)) e^{\int_{0}^{t}{V(\omega(z))dz}} \right) \qquad \mbox{for all}~t \geq 0.$$
Let us now assume that $x_0$ is a point at which $u$ has an extremum, i.e. $|u(x_0)| = \|u\|_{L^{\infty}}$. We
can assume that $u(x_0)$ is a maximum, the argument for a minimum is identical after multiplying the solution with $(-1)$. We introduce a function $p_{x_0}:\mathbb{R}_{+} \rightarrow [0,1]$ via
$$p_{x_0}(t) := \mbox{likelihood that a Brownian motion started in}~x_0~\mbox{hits the boundary within time $t$}.$$
This function is well-defined (alternatively, one could define $p_{x_0}$ as the point evaluation of the solution of a specific heat equation, see \cite{stein1}).
An application of the Feynman-Kac formula gives
\begin{align*}
 \|u\|_{L^{\infty}} = u(x_0) &=  \mathbb{E}_{x_{0}}\left(u(\omega(t)) e^{\int_{0}^{t}{V( \omega(z))dz}} \right) \\
&\leq  \|u\|_{L^{\infty}(\Omega)} \mathbb{E}_{x_{0}}\left(e^{\int_{0}^{t}{V( \omega(z))dz}} \right)  \leq   \|u\|_{L^{\infty}} (1- p_{x_0}(t))  e^{t \|V\|_{L^{\infty}}},
\end{align*}
and therefore, assuming $u$ to not be identically zero $\|u\|_{L^{\infty}(\Omega)} > 0$,
$$  (1-p_{x_0}(t)) e^{t \|V\|_{L^{\infty}}} \geq 1.$$
We observe that this argument gives the same conclusion more generally for positive maxima ($\forall x \in \Omega: u(x) \leq u(x_0)$ and $u(x_0) > 0$) and negative minima ($\forall x \in \Omega: u(x) \geq u(x_0)$ and $u(x_0) < 0$) -- the maximum of $|u|$ is always either a positive maximum or a negative minimum.
Parabolic rescaling allows us to assume that $d(x_0, \partial \Omega) = 1$ in which case it suffices to show that $\|V\|_{L^{\infty}(\Omega)} \geq c_1 > 0$ for some universal $c_1$. 
 We will now show that
$$p_{x_0}(1) \geq c_2 \, , \qquad \mbox{for some universal}~c_2 > 0 \, ,$$
which combined with the equation above, implies 
$$ \|V\|_{L^{\infty}} \geq \log{\left(\frac{1}{1-c_2}\right)} \, , 
\quad \mbox{from which the result follows.}$$
\begin{center}
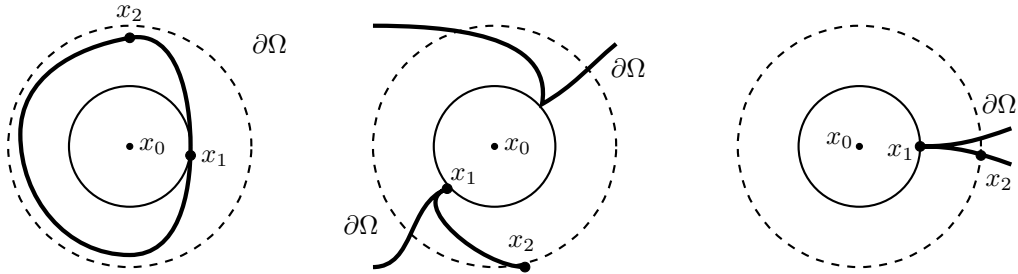
\begin{figure}[h!]
\begin{tikzpicture}[scale =0.8]
\filldraw (0,0) circle (0.05cm);
\draw [thick] (0,0) circle (1cm);
\draw [thick, dashed] (0,0) circle (2cm);
\draw [ultra thick] (-2, 2) to[out=0, in=80] (0.78, 0.7) to[out = 30, in =220] (2, 1.7);
\draw [ultra thick] (-2, -2) to[out=0, in=200] (-0.78, -0.7) to[out = 200, in =180] (0.5, -2);
\filldraw (-0.78,-0.7) circle (0.08cm);
\node at (-0.5, -0.5) {$x_1$};
\filldraw (0.5,-2) circle (0.08cm);
\node at (0.45, -1.6) {$x_2$};
\node at (0.38, 0) {$x_0$};
\node at (2.2, 1.3) {$\partial \Omega$};
\node at (-2.2, -1.3) {$\partial \Omega$};

\filldraw (-6,0) circle (0.05cm);
\draw [thick] (-6,0) circle (1cm);
\draw [thick, dashed] (-6,0) circle (2cm);
\draw [ultra thick] (-5, 0) to[out=90, in=10] (-6, 1.8) to[out = 190, in =90] (-7.8, 0.2)  to[out = 270, in =180] (-6, -1.8) to[out = 0, in =270] (-5, 0);
\filldraw (-5,-0.15) circle (0.08cm);
\node at (-4.6, -0.2) {$x_1$};
\filldraw (-6,1.8) circle (0.08cm);
\node at (-6, 2.2) {$x_2$};
\node at (0.68+5, 0.1) {$x_0$};
\node at (0.68+6, -0.1) {$x_1$};
\filldraw (7,0) circle (0.08cm);
\node at (2.3+6, 0.7) {$\partial \Omega$};
\filldraw (8,-0.15) circle (0.08cm);
\node at (8.3, -0.6) {$x_2$};

\filldraw (6,0) circle (0.05cm);
\draw [thick] (6,0) circle (1cm);
\draw [thick, dashed] (6,0) circle (2cm);
\draw [ultra thick] (8.5, 0.3) to[out=200, in=0] (7, 0) to[out = 0, in =160] (8.5, -0.3);
\node at (0.38-6, 0) {$x_0$};
\node at (2.3-6, 1.7) {$\partial \Omega$};
\end{tikzpicture}
\caption{The point of maximum $x_0$, the circle with radius $d(x_0, \Omega)$, the circle with radius $2 d(x_0, \Omega)$ (dashed) and the possible local geometry of $\partial \Omega$.}
\end{figure}
\end{center}
This is where the assumption of $\Omega$ being simply connected enters. The biggest disk that is centered at $x_0$ and is fully contained in $\Omega$ has radius $d(x_0, \partial \Omega) = 1$.
We put a circle of radius $2$ around $x_0$ and try to understand the form of $\partial \Omega$ in the annulus (see Fig. 4). The
assumption of $\Omega$ being simply connected implies that a long segment of the boundary $\partial \Omega$ is
contained in the annulus: more precisely, there are at least two points $x_1, x_2 \in \partial \Omega$ such that
$$ d(x_0, x_1) \leq 2 \qquad d(x_0, x_2) \leq 2 \qquad d(x_0, x_1) \geq 1$$
such that the $\partial \Omega$ is fully contained in the annulus between $x_1$ and $x_2$.
 The estimate $p_{x_0}(1) \geq c_2$ for some universal $c_2 > 0$ then follows from standard probabilistic estimates.
\end{proof}

\subsection{Proof of Theorem 2.}  
We first establish that the relevant property of location of extrema generalizes from Laplacian eigenfunctions to general Schr\"odinger equations.

\begin{lemma} 
Let $\Omega \subset \mathbb{R}^n$ be an open domain and let 
$u:\Omega \rightarrow \mathbb{R}$ satisfies
$ -\Delta u = Vu$ with Dirichlet conditions. If $|u(x_0)| = \|u\|_{L^{\infty}(\Omega)}$, then the likelihood of a Brownian motion
started in $x_0$ impacting on the boundary $\partial \Omega$ within time $t$ is 
$$p_{x_0}(T) := \mathbb{P} \left( \omega(t) \in \partial \Omega \quad ~~\hspace{3pt}~\mbox{for some}~\hspace{1pt}~~0 \leq t \leq T \right) \leq  1 - e^{-T \|V\|_{L^{\infty}}}.$$
\end{lemma}
\begin{proof} The proof is a straightforward adaption from \cite{stein1}. As in the proof of Theorem 1, we have
$$ \|u\|_{L^{\infty}} = u(x_0) =  \mathbb{E}_{}\left(u(\omega(T)) e^{\int_{0}^{T}{V( \omega(z))dz}} \right) \leq (1- p_{x_0}(T)) \|u\|_{L^{\infty}}  e^{T \|V\|_{L^{\infty}}}$$
and therefore
$$ p_{x_0}(T) \leq 1 - e^{-T \|V\|_{L^{\infty}}}.$$
\end{proof}
We can now invoke the argument of Georgiev \& Mukherjee \cite{georg} (their argument only uses that Brownian motion started in the point of maximum
has a quantitatively controlled likelihood of hitting the boundary, which we have just established in the more general context): we present their argument in abbreviated form: if a large part of $B$ was outside of $\Omega$, the likelihood of hitting the boundary would be large but we just established that it is not.
\begin{proof}[Proof of Theorem 2]   We pick a suitable
ball $B$ centered at $x_0$ with radius $r$, where
$$ r = c \left\| \frac{\Delta u}{u} \right\|^{-1/2}_{L^{\infty}(\Omega)} \, ,$$
and study the behavior of $B_{1} = \Omega \cap B$ and $B_2 = B \setminus \Omega$. 
The crucial idea
is to invoke the 2-capacity of $B_2$. A hitting inequality of Grigor'yan \& Saloff-Coste \cite{grig} implies that
$$ \mbox{cap}(B_2) \leq p_{x_0}(t) r^{n-2}.$$
The previous Lemma implies that the hitting probability can be made smaller than any fixed $\varepsilon > 0$ by making the radius smaller. 
Conversely, the isocapacitory inequality (see e.g. Maz'ya \cite[Section 2.2.3]{mazya}) implies that
$$ |B_2|^{\frac{n-2}{n}}  \leq \mbox{cap}(B_2).$$
This then implies that $|B_2| \leq \varepsilon |B_1|$ (after possibly changing the value of $c$ in a way that only depends on $\varepsilon$).
\end{proof}

\subsection{Proof of Corollary 2.}
\begin{proof} The torsion function $v(x_0)$ has a probabilistic interpretation (again via the Feynman-Kac formula)
as giving the expected lifetime of Brownian motion started in $x_0$ before it hits the boundary. Our main result 
implies that if the first Laplacian eigenfunction has a maximum at $x_0 \in \Omega$, then
$$   \inf_{y \in \partial \Omega}{ \| x_0 - y\|}  \geq c\cdot \lambda_1(\Omega)^{-1/2},$$
where $c$ is a universal constant. Since Brownian motion (generically) moves within a ball of radius $\sim \sqrt{t}$ within
time $\sim t$, this implies
$$ v(x_0) \geq c^*  \lambda_1(\Omega)^{-1}$$
for some different universal constant $c^{*} > 0$ that only depends on $c$. A precise bound on $c^{*}$ (depending
on $c$) could be derived using the reflection principle (see \cite{stein1}) but this is of no importance here.
 The result then follows from
$$ \|v\|_{L^{\infty}} \leq 4\cdot\mbox{inradius}(\Omega)^2 \leq \frac{4 j_0^2}{ \lambda_1(\Omega)},$$
where the first inequality can be found in \cite{ban3, ban2} and the second inequality (with $j_0 \sim 2.4\dots$ being the smallest
positive zero of the Bessel function) follows from putting a suitable test function inside the inradius.
\end{proof}
It is clear that these standard elliptic arguments are not sufficiently refined to explain why the constant in Corollary 2 is so close to 1: more sophisticated arguments are required.

\section{Remarks and Comments}

\subsection{Dynamical interpretation} There is a dynamical interpretation of Theorem 1 that may prove to be another avenue towards sharper bounds. Consider
a function satisfying $-\Delta u = Vu$ and let $x_0$ be a maximum.

\begin{center}
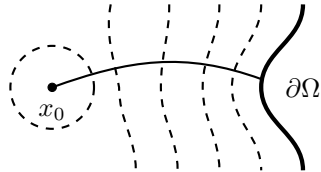
\begin{figure}[h!]
\begin{tikzpicture}[scale =1.1]
\filldraw (0,0) circle (0.05cm);
\draw [ultra thick] (3, 1) to[out=270, in=90] (2.5, 0) to[out = 270, in =90] (3, -1);
\node at (0,- 0.3) {$x_0$};
\draw [dashed, thick] (2.5, 1) to[out=270, in=90] (2.15, 0) to[out = 270, in =90] (2.5, -1);
\draw [dashed, thick] (2, 1) to[out=270, in=90] (1.8, 0) to[out = 270, in =90] (2, -1);
\draw [dashed, thick] (1.4, 1) to[out=270, in=90] (1.3, 0) to[out = 270, in =90] (1.5, -1);
\draw [dashed, thick] (0.7, 1) to[out=260, in=100] (0.8, 0) to[out = 270, in =90] (1, -1);
\draw [dashed, thick] (0.5, 0) to[out=270, in=0] (0,-0.5) to[out = 180, in =270] (-0.5, 0) to[out = 90, in =180] (0,0.5) to[out = 0, in =90] (0.5, 0);
\draw [thick] (0, 0) to[out=20, in=160] (2.5, 0.1);
\node at (3, 0) {$\partial \Omega$};
\end{tikzpicture}
\caption{The point of maximum $x_0$, level sets and the gradient flow towards $\partial \Omega$.}
\end{figure}
\end{center}
One would assume that, if $x_0$ is very close to the boundary, then the arclength-normalized gradient flow $\gamma: \mathbb{R}_{+} \rightarrow \Omega$ given by
$$ \dot \gamma(t) = - \frac{\nabla u(\gamma(t))}{\left| \nabla u(\gamma(t)) \right|} \, ,$$
will flow rather quickly to the boundary without getting trapped in critical points. Let us then consider the function $y(t) = u(\gamma(t))$. In two dimensions, the Laplacian has an interpretation with respect to level sets
$$ \Delta u = u_{\nu \nu} + \kappa u_{\nu},$$
where $u_{\nu}$ and $u_{\nu \nu}$ are the first and second derivative in the direction $\nabla u$ and $\kappa$ is the curvature
of the level set. Ignoring the curvature of the level set and assuming the gradient flow to move relatively quickly to the boundary,
we see that the equation for $y(t)$ can be approximated by
\begin{align*}
 - y''(t) &= V(\gamma(t)) y(t) \\
y'(0) &= 0 \\
y(0) &= \|u\|_{L^{\infty}(\Omega)}.
\end{align*}
A standard comparison inequality argument yields that
$$ y(t) \geq \cos{ \left(\|V\|^{1/2}_{L^{\infty}}t \right) } \cdot y(0).$$
This means that $y(t)$ requires $ \|V\|^{1/2}_{L^{\infty}}t \geq \pi/2$, which requires $t \sim \|V\|^{-1/2}_{L^{\infty}}$. It is not
clear to us whether any argument in this direction could be made quantitative.

\subsection{Barta's inequality.} Our inequality may also be regarded as a refinement of a classical inequality of Barta \cite{barta}, which states that for any $u \in C^2(\Omega, \mathbb{R})$ that vanishes on the boundary $\partial \Omega$
$$ \lambda_1(\Omega) \leq \sup_{x \in \Omega}{\frac{ -(\Delta u)(x)}{u(x)}} \leq \left\| \frac{\Delta u}{u}\right\|_{L^{\infty}(\Omega)}.$$
Barta's inequality implies that the best lower bound we can get on
$$   \inf_{y \in \partial \Omega}{ \| x_0 - y\|} \geq c \left\| \frac{\Delta u}{u}\right\|^{-1/2}_{L^{\infty}(\Omega)}   \qquad \mbox{is bounded from above by} \quad c \cdot \lambda_1(\Omega)^{-1/2},$$
which is the sharp result since
$$\mbox{inradius}(\Omega) \leq \frac{j_0}{ \lambda_1(\Omega)^{1/2}}.$$
Conversely, the simple inequality $\mbox{inradius}(\Omega) \leq  j_0 \lambda_1(\Omega)^{-1/2}$ could be used in conjunction with our inequality to obtain
$$   c \left\| \frac{\Delta u}{u}\right\|^{-1/2}_{L^{\infty}(\Omega)} \leq \inf_{y \in \partial \Omega}{ \| x_0 - y\|} \leq \mbox{inradius}(\Omega) \leq \frac{j_0}{ \lambda_1(\Omega)^{1/2}},$$
which implies Barta's inequality (with a non-optimal constant) and can be used to derive a stability version of Barta's inequality: if the distance of the maximum of $u$ to the boundary
is not on the same scale as the inradius, then Barta's inequality cannot be close to being sharp.

\subsection{Bounds for the constant.} It might be of interest to
get a good understanding of the optimal constant $c_1 > 0$ in 
$$   \inf_{y \in \partial \Omega}{ \| x_0 - y\|}  \geq c_1 \left\| \frac{\Delta u}{u} \right\|^{-1/2}_{L^{\infty}(\Omega)}.$$
As above, by specializing to $u$ being the first eigenfunction of the Laplacian and bounding the distance of the extremum to the boundary from above by the inradius, we get, for some $c_2 \geq c_1$,
$$ \inf_{\Omega ~\mbox{\SMALL simply connected}} ~~~~~\mbox{inradius}(\Omega) \sqrt{\lambda_1(\Omega)} \geq c_2,$$
where $c_2$ satisfies $c_2 \leq 1.445$ \cite{brown}. Our explicit construction (see Fig. 3) yields $c_1 \leq 1.37$. Lower bounds for these types of problems have always been more difficult \cite{ban2, croke,hayman,hersch,makai,osserman,protter}, the
currently best result $c_2 \geq 0.78$ being due to Banuelos \& Carroll \cite{ban2}. The purpose of this section is to
sketch how one could make our result $c_1 > 0$ quantitative. 

\begin{proposition} The straightforward quantization of our argument can at most achieve
$$ c_1 \geq 0.1.$$
\end{proposition}
We can reformulate 
 $$  (1-p_{x_0}(t)) e^{t \|V\|_{L^{\infty}}} \geq 1 \qquad \mbox{as} \qquad \|V\|_{L^{\infty}(\Omega)} \geq \sup_{t > 0}{ \frac{1}{t} \log{\left(\frac{1}{1-p_{x_0}(t)}\right)} }.$$

Parabolic rescaling allows us to set $d(x_0 , \partial \Omega) = 1$.  We will focus on the case, where the boundary of $\partial \Omega$ behaves like a straight line moving away from $x_0$. 
\begin{center}
\begin{figure}[h!]
\begin{tikzpicture}[scale=0.6]
\filldraw (4,0) circle (0.05cm);
\draw [thick] (4,0) circle (1cm);
\draw [thick, dashed] (4,0) circle (2cm);
\draw [ultra thick] (6.5, 0.2) to[out=200, in=0] (5, 0) to[out = 0, in =160] (6.5, -0.2);
\node at (0.38+4, 0) {$x_0$};
\node at (2.3+4, 1.7) {$\partial \Omega$};
\draw [thick] (15.5,0) -- (18,0);
\filldraw (13,0) circle (0.05cm);
\node at (13, -0.4) {$(0,0)$};
\filldraw (15.5,0) circle (0.05cm);
\node at (15.5, -0.4) {$(1,0)$};
\filldraw (18,0) circle (0.05cm);
\node at (18, -0.4) {$(2,0)$};
\coordinate [label=left:] (x) at (13,0);
\coordinate [] (y) at (13,0);
\node (D) [name path=D,draw,circle through=(y),label=left:$B$] at (x) {};
\draw (13,0) \foreach \x in {1,...,100}{--++(rand*0.25,rand*0.25)};
\end{tikzpicture}
\caption{Brownian motion started in $(0,0)$ and running for time $t=1$. What is the likelihood of it crossing the interval $\left\{(x,0): 1 \leq x \leq 2\right\}$?}
\end{figure}
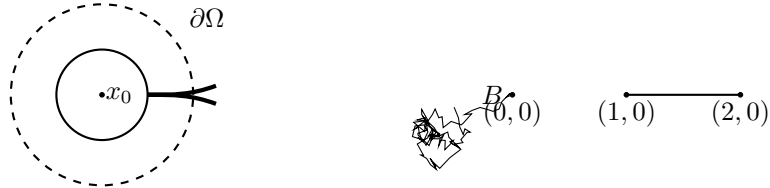
\end{center}
We invoke no quantitative information on the rest of the boundary and pretend that the boundary only consists of this line segment. 
A straightforward Monte-Carlo simulation gives that $p_{(0,0)}(1) \sim 0.1$.
This could be turned into a rigorous result if (a) the Monte-Carlo estimation was replaced by a rigorous computation (possibly by solving the associated heat 
equation), and (b) by establishing the following interesting  conjecture (or a slightly weaker quantitative variant).
\begin{quote} \textbf{Conjecture.} Let $\gamma: [0,1] \rightarrow \mathbb{R}^2$ be a smooth curve such that $\|\gamma(0)\| = 1$ and $\|\gamma(1) - \gamma(0)\| = 1$.
The likelihood of Brownian motion starting in $(0,0)$ and running for time $0 < T < \infty$ hitting the curve is minimized if and only if the curve is the straight line $\gamma(t) = (t+1,0)$ (or a
rotation thereof).
\end{quote}

One approach towards getting improved estimates might be the following: considering Fig. 5, it seems obvious that unless $\partial \Omega$ touches the ball $B(x_0, d(x_0, \partial \Omega))$
from multiple sides, the maximum is not likely to be achieved in $x_0$. Having multiple parts of the boundary touch the ball quickly leads to domains of the type that were used in
\cite{ban2, brown} to establish upper bounds.

\subsection{Rougher potentials.}  The presentation of our results is motivated by classical problems
related to Laplacian eigenfunctions. However, the approach certainly extends to rougher potentials. If we assume that
\begin{align*}
- \Delta u &= V u\quad \mbox{in}~\Omega \\
 u &= 0     \qquad \hspace{4pt} \mbox{on}~\partial\Omega,
\end{align*}
then the proper limit for our argument are applicability of the Feynman-Kac formula and 
$$h(t) \rightarrow 1 \qquad \mbox{where} \quad h(t) :=  \sup_{x \in \Omega}{  \mathbb{E}_{x}\left( e^{\int_{0}^{t}{|V(\omega(z))|dz}} \right) }$$
which naturally relates to Kato class via Khas'minskii lemma \cite{khas, simon}. Our argument implies  
$$ p_{x_0}(t) \leq 1 - h(t)^{-1}$$
which implies that the `wave-length' or radius of a ball that is approximately contained is given by $\sim \sqrt{T}$, where $T = \sup{ \left\{  t > 0: h(t) \leq 2 \right\}  }$.
If, for example, $h(t) \leq \exp{(t \|V\|_{L^{\infty}})}$, then we have $T \sim 1/\|V\|_{L^{\infty}}$ and recover our main result above. \\

\textbf{Acknowledgement.} The authors are grateful to Richard Laugesen for pointing out the connection to Barta's inequality.

\end{document}